\documentclass{article}
\usepackage{spconf,amsmath,amsthm,amssymb,amsfonts,graphicx,cite}

\def\*{\textsuperscript{$\!\ast$}}

\newtheorem{theorem}{Theorem}
\newtheorem{corollary}{Corollary}
\newtheorem{proposition}{Proposition}
\newtheorem{lemma}{Lemma}
\newtheorem{example}{Example}

\newtheorem{definition}{Definition}

\title{The Infinity of Randomness}
\name{Yongxin Li}
\address{School of Statistics and Mathematics, Central University of Finance and Economics}
\begin{document}
%
\maketitle
\begin{abstract}
This work starts from definition of randomness, the results of algorithmic randomness are analyzed from the perspective of application. Then, the source and nature of randomness is explored, and the relationship between infinity and randomness is found. The properties of randomness are summarized from the perspective of interaction between systems, that is, the set composed of sequences generated by randomness has the property of asymptotic completeness. Finally, the importance of randomness in AI research is emphasized.
\end{abstract}
\section{Randomness}
\subsection{Algorithmic Randomness}
 Randomness is one of the fundamental concepts of statistical and probability theory. The original motivation for defining randomness and studying the subsequent properties is to provide a foundation for probability theory. While the mathematical formalization of probability was generally solved by Kolmogorov, it's still unable to provide a proper formalization of randomness. Here we turn to algorithmic randomness, though the combination of algorithms and randomness appears to be incompatible and paradoxical, it presents reasonable after a deeper probing of their connection.
 
It was not until the beginning of the 20th century that mathematical definitions of randomness began to emerge. Attempting to provide an intuitively definition of probability based on a better understanding of randomness, Von Mises defined randomness in a sequence of observations in terms of unpredictability, he argued that a sequence should count as random if all ``reasonable" infinite subsequences satisfy the law of large numbers, the subsequences should have the same proportion of 0's and 1's in the limit under ``acceptable selection rules", and the notion of collective is defined on the random sequence.\cite{von1919grundlagen} Church perfectly made the definition of collective of von Mises mathematically formalized and reduced the selections of subsequences to computability. \cite{church1940concept} 
A completely different breakthroughs in the task of offering a definition of randomness was proposed by Kolmogorov \cite{kolmogorov1965three} and, independently, Solomonoff \cite{solomonoff1964formal,Solomonoff1964} ,  Chaitin\cite{chaitin1966length}. Their idea was to measure the amount of information in finite strings using the algorithmic approach, actually, reduce the notion of randomness to the notion of complexity.  The basic properties of Kolmogorov complexity were further developed in the 1970s. Independently, C. P. Schnorr \cite{SCHNORR1973376} and L. Levin \cite{levin1973notion} found a link between complexity and the notion of algorithmic randomness by introducing the so-called ``monotone complexity". And complexity turned out to be an important notion both in theory of computation and probability. 
Around the same time, Martin-L{\"o}f \cite{martin1966definition} developed a measure-theoretic approach to the notion of random sequences. The idea behind this approach is to identify the notion of randomness with the notion of typicalness, and random sequences should be statistically typical.
Since Martin-L{\"o}f's  definition in the setting of higher recursion theory, a hierarchy of algorithmic randomness notions has been developed, corresponding to to certain aspects of our intuition. \cite{downey2006calibrating,downey2010algorithmic,nies2012computability}

Though the defitions of randomness above focus on different aspects of randomness, they proposes algorithmic criteria distinguish random phenomena from non-random ones, and essentially they are equivalent to each other.  In the 1970s, Schnorr  and Chaitin \cite{chaitin1975randomness} found that Martin-L{\"o}f randomness is equivalent to the randomness defined by program-size complexity in describing random infinite binary sequences. Moreover, A further connection has been found that any Kolmogorov-Chaitin (Martin-L{\"o}f ) random sequence is also Mises-Wald-Church random.\cite{khrennikov2016probability} 

\subsection{Random Number Generators}

Though randomness stands alone, away from the human's systematic understanding of the world, we've managed to learn how to use randomness to study and explore the world, such as the usage of random numbers. Random numbers play important role in cryptographic applications, in numerical simulations of complex physical, chemical, biological and statistical systems.

Generally, there are two types of random number generators (RNG). They are true RNGs and pseudo-RNGs(PRNG). The true RNGs are based on some physical processes bonded with the environment or genuine quantum features. The pseudo-RNGs rely on the output of a deterministic algorithm with a shorter random seed. Pseudo RNGs rely on the output of a deterministic algorithm with a short random seed. 

PRNGs are widely used in applications related to Monte Carlo methods. Named by Drs. Stanislaw Ulam, John von Neumann, Nicholas Metropolis, etc, in the 1940's, Monte Carlo methods  is a numerical method that makes use of random numbers to solve mathematical problems, for which an analytical solution is not available. \cite{metropolis1949monte} 

As deterministic algorithms, PRNGs seem reliable under the certain tests, but because of the finite period and the limitation of the theoretical basis, they usually have undetected long range correlations and other unexpected defects that will result in wrong solutions.\cite{marsaglia1968random,de1988parallelization,de1990long,kalle1984problems,ferrenberg1992monte,ossola2004systematic,grassberger1993correlations} 
 It has been realized that good PRNG should have such qualities: sound theoretical basis, long period, ability to pass empirical tests, and efficient, repeatable, and portable.\cite{gentle2003random}

\section{Source of Randomness and Intelligence}
\subsection{Interactions between Subsystems}

 In direct cognitional terms, randomness is defined in relation to human knowing, whether the human scientist can reach a systematic understanding of all that he can know, in the present and future. \cite{mcshane1970randomness}   Being a subsystem of the universe, we humans gain our knowledge of the universe through observation, measurement, induction, and reasoning, and realize the deep mathematical principles on which the universe runs. 

After the establishment of Newtonian mechanics, people once thought that if someone knows the initial state of every atom in the universe,  the values of position and velocity for any given time in the past and future can be completely calculated from  the laws of Newtonian mechanics. \cite{pierre2007philosophical} However,  in his research on the three-body problem, Poincaré discovers a chaotic deterministic system, that given the law of gravity and the initial positions and velocities of the only three bodies in all of space, though the subsequent positions and velocities are fixed, but appear random because of  our inability to have infinite precision and infinite power. \cite{poincare1967new}

Developed in the early decades of the 20th century,  quantum mechanics introduces an element of unpredictability or randomness in the scientific description of subatomic particles. Different from classical physics, in quantum mechanics energy, momentum, angular momentum, and other quantities of a bound system are restricted to discrete values. The accuracy of physical quantities, such as position and velocity of subatomic particles, interacts, illustrating the absence of our ability to have infinite precision and infinite power.  Even more, theoretical predictions, confirmed experimentally, such as the violation of Bell inequalities, point to an affirmative answer that our universe emerges as intrinsically non-deterministic and unpredictable.\cite{bell1964einstein,bell2004speakable,hensen2015loophole}

As for the randomness in quantum physics, parallel to the contemporary theory of quantum measurements, and collapse of the wave function, from many physicists' point of view,  the sources of randomness originate from the interaction between the observed system and environment.  \cite{wheeler2014quantum,zurek2003decoherence,zurek2009quantum} It is logical to assume that the ``system" and the environment in consideration are subsystems of a grand system (maybe the universe), where many subsystems operate according to their rules and interact with each other. Such interactions, especially in a large number of subsystems, create an extremely complex behavior that can only be studied in terms of randomness. As observers, comprehenders and explainers, we human beings should also be one of the subsystems, and randomness appearing in the process of observations and measurements should be outcome of interactions among the subsystems.

The process of interactions shows us a algorithmic approach to randomness. The ``system"  could also be thought of as a mathematical system defined with a set of simple rules or equations. We can choose a series of mathematical systems, measure and collect their interacting behavior to produce randomness.  We should keep in mind that behind all emergence of randomness exists the notion of infinity. Randomness comes behind the absence of infinite knowledge, infinite power or the interplay of infinite agents. The planet we live in exists in four-dimensional space-time, and infinity is the intrinsic property of time and space, making the emergence of randomness, rules, and everything around us possible.

\subsection{Randomness and Artificial Intelligence}

\textit{Why do you insist that the Universe is not a conscious intelligence, when it gives birth to conscious intelligences?}  (Cicero, c. 44 bce) \cite{Berezin2018isotopic}

We have the ability to comprehend the internal workings of atoms,  the nature of black holes and G{\"o}del’s  incompleteness theorem. But does the universe be structured such as to be totally comprehensible to a subsystem of itself?
Human beings are more than mere observers; we are also comprehenders and universal explainers. \cite{deutsch2011beginning}

Starting with the work of Turing in 1950 \cite{10.1093/mind/LIX.236.433} Artificial Intelligence (AI) technology and its related applications become part of people's daily life. The great success of AI technology such as machine learning, data mining, computer vision, natural languages processing, ontological-based search engine, has enabled machines to play chess, paint, drive vehicles, and even chat with people. As R. Penrose  pointed to the transcendental nature of human mind\cite{10.1093/oso/9780198519737.001.0001},  human mind seems not to be driven by computer programs. Artificial Intelligence presents a new forum within which to address a large number of philosophical questions, both classical and novel, about the nature of life, intelligence, and existence. As for AI, debates  usually get hopelessly mired down within the seemingly bottomless pit of subjective concepts like ``mind" and ``consciousness."

The situation forces us to rethink what it is to be ``alive", as well as ``intelligent". To understand intelligence, a natural idea is to study ourselves. It's impossible to list strictly behavioral criteria for the ``intelligence", part of the failure of AI to achieve ``intelligence" is due to the fact that it is setted chasing a moving target. It has been thought that playing chess required true intelligence, until computers can beat best players of humans. The criteria of ``intelligence" lie in the process of life.

Life is a planetary process in the context of space and time, the emergence and evolution of life is deeply influenced by the environment of our planet. Interacting with the environment, living systems have developed different orders at many scales, in terms of composition, spatial configuration, and dynamics. And living states are determined at the microscale by the quantum mechanics of atomic and molecular orbitals, and are governed by the orbital-scale dynamics of reactions and the physical chemistry of molecular assemblies.  \cite{smith2016origin}  
 
A living system is comprised of many subsystems operating at vast ranges of space and time scales.   Since all subatomic particles obey quantum mechanics and all matter is made out of these particles, quantum mechanics plays a crucial role in  a broad range of complex and dynamic biological processes and functions.\cite{kim2021quantum}  

In other words, intrinsic randomness can propagate from the microscopic constituents of matter to macroscopic world and act in the living systems, and to make AI more ``intelligent", it's necessary to offer AI the ability to grow and evolve from dealing with randomness phenomena.

\section{Projection of Infinity}
\subsection{Properties of randomness}
\begin{definition}
\textbf{System}(S): a group of physical or logical,  interacting or interrelated elements with properties can be observed or measured; \textbf{Subsystem}(SS): part of the whole system that can be treated as a unity.
\end{definition}

\begin{example}
In the game of playing dice, the whole system consists of the dice, the players, and the environment. The number of dots on the upper face is the property to be observed. The dice, the players and the environment are the subsystems of the whole system.
\end{example}

\begin{example}
The whole instruction set of a electronic computer is a system, and the instructions used to produced sequences of pseudorandom numbers compose a subsystem, the sequences of pseudorandom numbers is the property to be observed.
\end{example}

\begin{lemma}[Strong structure theorem, \cite{tao2007structure}] \label{lem-tao}
Let H be a real finite-dimensional Hilbert space, $S = \{v: \| v \| _{H} \leqslant 1, \forall  v \in S\} \subset H$, let $\epsilon >0$, and let $F: \mathbf{Z}^{+} \rightarrow \mathbf{R}^{+}$ be an arbitrary function. Let $f\in H$ be such that $\| f \| _{H} \leqslant 1$. Then we can find an integer $M = O_{F, \epsilon}(1)$ and a decomposition $$f = f_{str} + f_{psd} + f_{err}$$ where $f_{str}$ is a $(M, M)$ structured vector that there exits a decomposition $f_{str} = \sum_{1 \leqslant i \leqslant M} c_{i}v_{i}$ (for $v_{i} \in S$ and $c_{i} \in [-M, M], \forall 1\leqslant i \leqslant M$), $f_{psd}$ is a $1/F(M)$ pseudorandom vector, and $f_{err}$ is a vector that has norm at most $\epsilon$, $O_{F, \epsilon}(1)$ denotes some quantity depending only on F and $\epsilon$ .
\end{lemma}

\begin{proposition}
There exit systems with true randomness, and the randomness can be extracted with observable properties of logical systems. 
\end{proposition}

\begin{proof}
Since all subatomic particles obey quantum mechanics and all physical systems are made out of these particles, the physical systems should have properties of true randomness naturally.

According to Lemma \ref{lem-tao}, we can have a segmentation of logical systems in terms of $f_{str}$, and let the randomness affect the segmentation, the output will be random.
\end{proof}

\begin{definition}
\textbf{Random system} (RS): subsystems that the true randomness properties of which can be measured properly. 
\end{definition}

The randomness of RS can be extracted with measurement systems. Let designate a measurement $m: \rm{RS} \circ \rm{SS} \rightarrow \mathbf{Z/}q\mathbf{Z}$, that divides the observable properties of SS equally into q segments and correlates the randomness of RS with the segmentation. If we take the measure n times, we'll get a sequence $r_{in}:=m_{1}m_{2}...m_{n}, 1\leqslant i \leqslant \infty$. Let $R_{n} = \{r_{in}\}$ be set of all the output sequences of n-times measurement.

\begin{theorem} \label{the-main}
As $n \rightarrow \infty$, the cardinality of $R_{n}$ is the cardinality of the set of real number $\mathbf{R}$.
\end{theorem}

\begin{proof}
Without loss of generality, we label the element of $R_{n}$ as $0.m_{1}m_{2}...m_{n}$, and let $\rm{q} =10$. So we have $0 \leqslant 0.r_{in} \leqslant 1,  n \rightarrow \infty$, i.e. $\exists r_{i} \in [0,1], s.t. \ r_{i} = 0.r_{in}$.

And for $\forall r_{i} \in [0,1]$, we can find a sequence $r_{in} \in R_{n}, s.t. \ 0.r_{in} =r_{i}, n \rightarrow \infty$.

So there exits a bijection map between $R_{n}$ and $[0,1]$ as $n \rightarrow \infty$.
\end{proof}

Theorem \ref{the-main} reveals the equivalence of randomness and space-time, and explains the efficiency of statistical methods.

\begin{definition}
\textbf{Asymptotic Completeness}: If As $n \rightarrow \infty$, the cardinality of $R_{n}$ is the cardinality of the set of real number $\mathbf{R}$, $R_{n}$ has the property of asymptotic completeness.
\end{definition}

\begin{corollary}
Randomness extracted from different physical systems is equal in terms of asymptotic completeness.
\end{corollary}

\begin{definition}
\textbf{Pseudorandomness}: The property of logical systems that exclude structural properties.
\end{definition}

\begin{theorem} \label{the-com}
The combination of randomness and pseudorandomness is randomness.
\end{theorem}

\begin{proof}
It's trivial.
\end{proof}

In 1975, Chaitin found a profound relationship between randomness and pseudorandomness of the output sequences of measurements. In other words, pseudorandomness can be treated as a special case of randomness.

\begin{lemma} \label{lem-chai} \rm{\cite{chaitin1975randomness,chaitin1975theory,zurek1989algorithmic}}
It impossible to prove  the randomness of a "random-looking" long string is a natural consequence of the algorithmic definition of randomness.
\end{lemma} 

So when the output strings of a PRNG are tested to be random enough for its application, the PRNG can be accepted. In addition, the randomness of the strings can be amplified in statistical or computational ways.

\begin{example}
{\rm{\cite{yao1982theory,levin1987one}} } Entropy extractor PRNGs are based on various mathematical theories of computer science, and  the XOR gate is one of the simplest. An XOR gate takes two binary values and outputs one. If the input is different, (0, 1) or (1, 0), the output is 1. If the input is the same, then we'll have 0. XOR can amply the computational unpredictability of output strings, when the inputs are independent instances.
\end{example}

\subsection{Applications}

The randomness phenomena in physical and logical systems are the outcome of observations and measurements, namely the results of interactions between subsystems. And we also realize that the true or intrinsic randomness can lead us to the infinity. 

In the applications of algorithmic randomness to design RNGs,  we come to the probability theory paradox \cite{shen2017kolmogorov}, to satisfy the test criteria, such as those from NIST and Dieharder, events with small probability are restricted and even excluded. 

The pseudo RNG is a trade-off between computation efficiency and randomness. Theorem \ref{the-com} and Lemma \ref{lem-chai} ensure a way to more randomness. It's a natural way to design random number generators $G$ depending on the interactions of subsystems $SS$ and $RS$. Here we choose a logical subsystems $SS$, such as a pseudo RNG. With statistically equal segmentation, the output sequences of pseudorandom numbers can take on a role of digital dice when combined with the measurement results of $RS$.
$$
G := RS \circ SS
$$
\quad Moreover, randomness can help us do more in the field of AI. As Theorem \ref{the-main} shows the equivalence of randomness and space-time, if AI should be intelligent as human beings, randomness can help AI to grow and evolve in a natural way. 

\section{Conclusion}

In this work, we show different aspects of the theory of randomness and summarize the relationship between randomness and infinity in the sense that the set formed by randomly generated sequences has asymptotic completeness. For the application of randomness theory, we propose to use the combination of physical and logical systems to construct random number generators, so as to improve randomness and computational efficiency. At the same time, we believe that randomness is closely related to the development of artificial intelligence technology.

\bibliographystyle{IEEEbib}
\bibliography{Randomness}

\end{document}